\documentclass{amsart}
\usepackage{amssymb,amsfonts,latexsym}

\numberwithin{equation}{section}

\newtheorem{theorem}{Theorem}[section]
\newtheorem{lemma}[theorem]{Lemma}
\newtheorem{proposition}[theorem]{Proposition}
\newtheorem{corollary}[theorem]{Corollary}
\theoremstyle{definition}

\newtheorem{question}[theorem]{Question}

\newtheorem{remark}[theorem]{Remark}

\newtheorem{problem}[theorem]{Problem}

\newtheorem *{Theorem A}{Theorem A}
\newtheorem *{Corollary B}{Corollary B}

\newcommand{\ach}{\check{A}}
\newcommand{\C}{\mathbb{C}}

\newcommand{\la}{\langle}
\newcommand{\ra}{\rangle}
\newcommand{\Z}{\mathbb{Z}}
\begin{document}

\title
{Isotropy in group cohomology}

\author{Nir Ben David}

\address{Department of Mathematics, Technion-Israel Institute of Technology, Haifa 32000, Israel}
\email{benda@tx.technion.ac.il}

\author{Yuval Ginosar}

\address{Department of Mathematics, University of Haifa, Haifa 31905, Israel}
\email{ginosar@math.haifa.ac.il}

\author{Ehud Meir}
\address{Department of Mathematical Sciences, University of Copenhagen, Universitetsparken 5, DK-2100, Denmark}
\email{meirehud@gmail.com}

\date{\today}


\begin{abstract}
The analogue of Lagrangians for symplectic forms over finite groups is studied, motivated by the fact that
symplectic $G$-forms
with a normal Lagrangian $N\triangleleft G$ are in one-to-one correspondence, up to inflation,
with bijective 1-cocycle data on the quotients $G/N$.
This yields a method to construct groups of central type from such quotients, known as
Involutive Yang-Baxter groups.
Another motivation for the search of normal Lagrangians comes from a non-commutative generalization of
Heisenberg liftings which require normality.

Although it is true that symplectic forms over finite nilpotent groups always admit Lagrangians,
we exhibit an example where none of these subgroups is
normal. However, we prove that symplectic forms over nilpotent groups always admit normal Lagrangians if all their
$p$-Sylow subgroups are of order less than $p^8$.
\end{abstract}

\maketitle
\begin{section}{Overview}
\subsection{}
Let $V$ be a vector space over a field $F$ and let $W\subset V$ be any subspace.
Here is a natural way to construct an alternating form on $V$ such
that $W$ is isotropic with respect to this form. Let $U:=V/W$, and
identify $U$ with a subspace of $V$ which is complementary to $W$.
Thus, every $v\in V$ is uniquely expressed as $v=w+u$, where $w\in
W$ and $u\in U$. For any $\varphi\in\check{W}:=$Hom$_F(W,F)$ and $w\in W$, we use the pairing
notation $\la \varphi,w\ra\in F$ to denote the evaluation of $\varphi$ at $w$.
Now, for any linear transformation $\pi:U\to \check{W}$
define
\begin{equation}\label{alphapi}
\begin{array}{rcl}
\alpha_{\pi}:V\times V &\rightarrow &F\\
(w_1+u_1,w_2+u_2)&\mapsto & \la\pi(u_1), w_2\ra-\la \pi(u_2),
w_1\ra.
\end{array}
\end{equation}
Then the form $\alpha_{\pi}$ is indeed alternating, with $W$
isotropic. More generally, for any
alternating form ~$\alpha':U\times U\to F$ on the quotient space
$U$, the form
\begin{equation}\label{alphapinf}
(w_1+{u_1},w_2+{u_2})\mapsto \alpha_{\pi}(w_1+{u_1},w_2+{u_2})+\alpha'(u_1,u_2)
\end{equation}
also satisfies the above requirements. The following properties
are easily verified: Firstly, $W$ is a maximal isotropic subspace
with respect to any of the alternating forms given in \eqref{alphapinf} if and only if
$\pi$ is injective. Secondly, in case dim$_FW$=dim$_FU$, then
none of the above forms admits a nontrivial radical if and only if $\pi$ is
bijective. In other words, for any subspace $W\subset V$ with
dim$_FW=\frac{1}{2}$dim$_FV$, equation \eqref{alphapinf} describes a method to
construct a symplectic form (i.e. admitting a trivial radical) on $V$ such that $W$ is a {\it
Lagrangian} with respect to this form.

Conversely, let $\alpha:V\times V \rightarrow F$ be an alternating
form with $W\subset V$ isotropic. Define a linear map
\begin{equation}\label{pialpha}
\begin{array}{rcl}
\pi_{\alpha}:U& \rightarrow& \check{W}\\
\la\pi_{\alpha}(u), w\ra&:= &\alpha({u},w).
\end{array}
\end{equation}
Then $W$ is maximal isotropic with respect to the form $\alpha$ if
and only if $\pi_{\alpha}$ is injective.
Moreover, $\alpha$ is symplectic if and only if $\pi_{\alpha}$ is
bijective, in particular, dim$_FW=\frac{1}{2}$dim$_FV$.

The maps \eqref{alphapinf} and \eqref{pialpha} may be considered as mutually inverse in the sense that
$\pi_{\alpha_{\pi}}=\pi$, and that $\alpha_{\pi_{\alpha}}$ differs from
$\alpha$ by an alternating form which is inflated from $U$.

\subsection{}
An analog is established for forms on groups. Let $G$ be a group.
A $G$-form over an abelian group $M$ is a map $\alpha:G\times G\to M$ such that for every $g\in G$, both
res$|^G_{C_G(g)}\alpha(g,-)$ and res$|^G_{C_G(g)}\alpha(-,g)$ are group homomorphisms from the centralizer $C_G(g)$ to $M$.
Given a $G$-form $\alpha$ and an element $x\in G$, it is convenient to use the following notation.
\begin{equation}\label{perp}
K_x=K_x(\alpha):=\{g\in C_G(x) | \alpha(g,x)=\alpha(x,g)=0\}.
\end{equation}
A subgroup $H<G$ is {\it isotropic} with respect to a form $\alpha:G\times G\to M$ if for every $h\in H$,
$C_H(h)= K_h$. In other words, if
the restrictions res$|^G_{C_H(h)}\alpha(h,-)$ and res$|^G_{C_H(h)}\alpha(-,h)$ to the homomorphisms from
$C_H(h)$ to $M$ are trivial for any $h\in H$.
A $G$-form $\alpha:G\times G\to M$ is
{\it symplectic} if, in addition,
\begin{itemize}
\item $\alpha(g,h)=-\alpha(h,g)$ for every $(g,h)\in G\times G$ such that $g$ and $h$ commute ($\alpha$ is alternating)
\footnote
{this condition may be slightly strengthened to $\alpha(g,g)=0$ for every $g\in G$.},
and
\item $K_g(\alpha)=C_G(g)$, that is res$|^G_{C_G(g)}\alpha(g,-)=0$, if and only if $g=e$ ($\alpha$ is non-degenerate).
\end{itemize}

Alternating forms on groups are naturally obtained from 2-cocycles.
Let $c\in Z^2(G,M)$ be a 2-cocycle with values in a trivial $G$-module $M$. Then
\begin{equation}\label{altform}
\begin{array}{rcl}
\alpha_c:G\times G &\rightarrow &M\\
(g,h)&\mapsto & c(h,g)-c(hgh^{-1},h).
\end{array}
\end{equation}
is an alternating form on $G$, called the alternating form {\it associated} to $c$ (see \cite[3.16]{BG}).
It is easy to show that if $g$ and $h$ commute,
then $\alpha_c(g,h)$ depends only on the cohomology class of $c$,
and not on the particular representative.
From now onwards our discussion is over cocycles which take their values in the multiplicative group $\C^*$
(or, alternatively, over cohomology classes in $H^2(G,\C^*)$)
rather than over arbitrary alternating $G$-forms.
In this regard, we say that a subgroup $H<G$ is isotropic with respect to $[c]\in H^2(G,\C^*)$ if
res$^G_H[c]=0\in H^2(H,\C^*)$.
\begin{remark}
The above definition of isotropy of a subgroup $H<G$ with respect to a 2-cocycle $c\in Z^2 (G,\C^*)$
says that the restriction of $c$ to $H$ is cohomologically trivial. However,
cohomologically nontrivial 2-cocycles $c\in Z^2 (G,\C^*)$ may give rise to trivial associated alternating forms $\alpha_c$.
Such cocycles belong to cohomology classes in $B_0(G)$ (the {\it Bogomolov multiplier}) in the sense of \cite{Bogo},
and are called {\it distinguished} in the sense of \cite{H90}.
\end{remark}

The non-degeneracy property for $G$-forms is much harder to attain.
A cocycle $c\in Z^2(G,\C^*)$ and its cohomology class $[c]\in H^2(G,\C^*)$ over the trivial $G$-module $\C^*$
are called {\it non-degenerate} if their associated alternating form $\alpha_c$ is symplectic.
This is the same as to say that for every $x\in G$ there exists an element $g\in C_G(x)$ such that
$c(x,g)$ and $c(g,x)$ are not equal.
Groups having a non-degenerate cohomology class are termed {\it of central type}.
These are groups of square order which admit ``large" projective representations (see Lemma \ref{tfae}).
By a deep result of R. Howlett and  I. Isaacs \cite{hi}, based on the classification of finite simple groups, it is
known that all such groups are solvable.
The following is a group-theoretic analogue of the fact that the dimension of a Lagrangian for a symplectic form on
a vector space $V$ equals $\frac{1}{2}$dim$_F(V)$.
\begin{proposition}\label{notnormal1}
Let $[c]\in H^2(G,\C^*)$ be a non-degenerate class. Then
any isotropic subgroup $H<G$ with respect to $[c]$ is of order dividing $\sqrt{|G|}$.
\end{proposition}
An isotropic subgroup $H<G$ with respect to a non-degenerate class $[c]\in H^2(G,\C^*)$ is a {\it Lagrangian}
with respect to this class if $|H|=\sqrt{|G|}$.
In particular, by Proposition \ref{notnormal1}, if a Lagrangian exists, then it is maximal isotropic.

In order to imitate the above vector space constructions in group-theoretic terms,
we assume that $G$ admits a normal subgroup $A$ of the same order as that of $Q:=G/A$ (in
particular $G$ is of square order). We seek after a non-degenerate class $[c]\in H^2(G,\C^*)$
with $A$ maximal isotropic.
Apparently, if $A\vartriangleleft G$ is isotropic with respect to a non-degenerate class
$[c]\in H^2(G,\C^*)$, then it is necessarily abelian (see Lemma \ref{abelian}).

The linear transformation $\pi:U\to \check{W}$ is replaced now by a 1-cocycle $$\pi:Q\to \ach.$$
Here $\check{A}:=$Hom$(A,\C^*)$ is endowed with the diagonal $Q$-action, which is induced from the action of $Q$ on
$A$ (by conjugation in $G$), namely
\begin{equation}\label{diag}
\la q(\chi),a\ra:=\la\chi,q^{-1}(a)\ra,
\end{equation}
for every $q\in Q, \chi\in \ach$ and $a\in A$.

To begin with, suppose that $A$ admits a complement in $G$, that is $G=A\rtimes Q$.
P. Etingof and S. Gelaki \cite{eg3} observed that any 1-cocycle
$\pi\in Z^1(Q,\ach)$ gives rise to a 2-cocycle
\begin{equation}\label{eg}
\begin{array}{rcl}c_{\pi}:G\times G&\to &\C^*\\
(a_1q_1,a_2q_2)&\mapsto &\la\pi(q_1),q_1(a_2)\ra^{-1}(=\la\pi(q_1^{-1}),a_2\ra),
\end{array}
\end{equation}
with an associated alternating form
(compare with \eqref{alphapi})
$$\alpha_{c_{\pi}}(a_1q_1,a_2q_2)=\la\pi(q_1),q_1(a_2)\ra\cdot\la\pi(q_2),q_2(a_1)\ra^{-1}, a_i\in A,q_i\in Q,$$
admitting $A$ as an isotropic subgroup.
Moreover, $\pi$ is bijective if and only if $c_{\pi}$ (and hence also $\alpha_{c_{\pi}}$) is non-degenerate.

Conversely, let $c\in Z^2(G,\C^*)$ with $A$ isotropic. Define
\begin{equation}\label{pi}
\begin{array}{rcl}
\pi_c=\pi_{[c]}:Q& \rightarrow& \ach\\
\la\pi_{c}(q), a\ra&:= &\alpha_c(q,a),
\end{array}
\end{equation}
for every $a\in A$ and any $q\in Q$. Then $\pi_c$ is a 1-cocycle (compare with \eqref{pialpha}).
Moreover, $\pi_c$ is bijective if and only if $c$ (or $\alpha_{c}$) is non-degenerate.

Again we obtain the mutually inverse property in
the sense that $\pi_{c_{\pi}}=\pi$, and that $c_{\pi_{c}}$ differs from $c$ by an alternating form inflated from $Q$.

We remark that groups admitting a bijective 1-cocycle,
namely {\it involutive Yang-Baxter} (IYB) groups,
are a key ingredient in the study of set-theoretic solutions of the quantum Yang-Baxter equation (see \cite{CJdR,ess}).

The correspondence between $G$-forms with $A\vartriangleleft G$ abelian isotropic
(modulo the $G$-forms inflated from $Q=G/A$)
and classes in $H^1(Q,\ach)$ still holds
even when the quotient $Q$ does not embed in $G$ as a complement of $A$,
though is more complicated. To formulate it we need the following notation.
For a group extension \begin{equation}\label{exten} [\beta]:1\to A\to G\to Q\to 1, \ \ [\beta]\in H^2(Q,A),
\end{equation} let
$$
{\rm res}^G_A:H^2(G,\C^*)\to H^2(A,\C^*),\ \ {\rm inf}^Q_G:H^2(Q,\C^*)\to H^2(G,\C^*)
$$
be the restriction
and inflation maps respectively.
Next, let
$$\mathcal{K}_{\beta}:=\{[\pi]\in H^1(Q,\ach)|\ \  [\beta] \cup [\pi]=0\in H^3(Q,\C^*)\}$$
be the subgroup of classes in $H^1(Q,\ach)$ annihilating the cup product with $[\beta]$. In \cite{BG} it is shown that
bijectivity is a cohomology class property, and that when $|A|=|Q|$, the non-degeneracy property is independent of the
representative modulo inflations.
We have
\begin{theorem}\label{main2} \cite[Theorem A]{BG}
Let \eqref{exten}
be an extension of finite groups, where $A$ is abelian.
Then there is an isomorphism
\begin{equation}\label{iso}
\mathcal{K}_{\beta}\simeq\ker({\rm res}^G_A) / [{\rm im}({\inf}^Q_G)].
\end{equation}
If, additionally, $|A|=|Q|$, then the isomorphism \eqref{iso} induces a 1-1 correspondence between bijective classes
in $\mathcal{K}_{\beta}$, and non-degenerate classes in $\ker({\rm res}^G_A)/[{\rm im}({\inf}^Q_G)]$.
\end{theorem}

Theorem \ref{main2} actually describes all groups of central type which have a non-degenerate form
that admits a normal Lagrangian.

When the extension \eqref{exten} splits, then certainly $[\beta] \cup [\pi]=0$
for every $[\pi]\in H^1(Q,A^*)$, and the correspondence in Theorem
\ref{main2} amounts to the one determined by \eqref{eg} and \eqref{pi}.

An intriguing question arises following Theorem \ref{main2}:
\begin{question}\label{q}
Let $[c]\in H^2(G,\C^*)$ be a non-degenerate class.
Does $[c]$ admit a Lagrangian, that is a maximal isotropic subgroup $A<G$ of order $\sqrt{|G|}$?
Moreover, does $[c]$ admit a normal (and hence abelian) Lagrangian $A\lhd G$?
\end{question}
If a non-degenerate class $[c]\in H^2(G,\C^*)$
gives an affirmative answer to Question \ref{q}, then by Theorem \ref{main2},
the corresponding quotient $G/A$ is an IYB group,
admitting a bijective 1-cocycle datum determined
by $[c]$.
\subsection{}\label{hlift} The motivation of Question \ref{q} stems also from lifting problems in classical
representation theory. The connection is briefly introduced hereby, and is discussed in more details in \cite{GOS}.

Let $N\vartriangleleft G$ be a normal subgroup and let
$\eta\in$Hom$(N,\C^*)^G$ be a 1-dimensional representation of $N$ which is stabilized by $G$, that is
$$\eta(n)=\eta(gng^{-1}), \forall n\in N, g\in G.$$
Then $(G,N,\eta)$ is a character triple (see \cite[page 186]{I}).
We say that a $G$-representation $\eta'$ {\it lies above} $\eta$
if the restriction of
$\eta'$ to $N$ admits $\eta$ as a constituent. Aspects of the following problem were thoroughly studied.
\begin{problem}
Describe the irreducible $G$-representations which lie above $\eta$.
\end{problem}
Next, the transgression map (see \cite[\S 1.1]{K2}) sends $\eta$ to a cohomology class tra$(\eta)\in H^2(G/N,\C^*)$,
which determines the obstruction for extending $\eta$ to a 1-dimensional representation of any intermediate group $N<H<G$
as follows.
Since the sequence
$$\text{Hom}(G,\C^*)\xrightarrow{\text{res}}\text{Hom}(N,\C^*)^G\xrightarrow{\text{tra}} H^2(G/N,\C^*)$$
is exact \cite[Theorem 1.1.12]{K2}, then $\eta$ is restricted from a morphism $\eta_0\in$ Hom$(H,\C^*)$ for every
subgroup $H/N<G/N$ which is {\it isotropic} with respect to tra$(\eta)$.
The morphism $\eta_0$ is determined up to the image of an inflation from $G/N$ (in $|(G/N)_{\text {ab}}|$ many ways).

The next step is to lift $\eta_0\in$ Hom$(H,\C^*)$ to a $G$-representation.
Suppose first that $G/N$ is a vector space over a finite field.
Then, as explained in \cite[Proposition 8.3.3]{BF}, the induction of $\eta_0$ from $H$ to $G$ is an irreducible
$G$-representation (of dimension $|G/H|=\sqrt{|G|}$)
if and only if the subspace $H/N$ is {\it maximal} isotropic with respect to tra$(\eta)$.
Moreover, if the form tra$(\eta)\in H^2(G/N,\C^*)$ is symplectic, then any maximal isotropic subspace $H$ is a Lagrangian.
In this case $|G/H|=|H/N|$ and so the irreducible $G$-representation ind$_H^G(\eta_0)$,
lying above $\eta$, depends neither on the choice of
$H$ nor on the extension $\eta_0\in$ Hom$(H,\C^*)$.
This procedure is termed {\it Heisenberg lifting} as its
guiding example yields a complete description of the
irreducible representations of the Heisenberg groups $\mathcal{H}_p$ as follows.
The $\mathcal{H}_p$-representations are computed using the central extension
$$1\to N\to \mathcal{H}_p \to \mathcal{H}_p/N \to 1,$$
where $N$ is cyclic of order $p$ and $\mathcal{H}_p/N$ is a 2-dimensional vector space over the field $\mathbb{F}_p$.
Since $N$ is central, all its representations are $\mathcal{H}_p$-invariant. The trivial and
the non-trivial $N$-representations are handled differently.
The trivial representation gives rise to a trivial class in $H^2(\mathcal{H}_p/N,\C^*)$
and therefore extends as a 1-dimensional representation all the way to $\mathcal{H}_p$, which is isotropic.
By that we obtain $p^2$ representations of dimension 1 for the Heisenberg group.
The other $p-1$ non-trivial representations of $N$ give rise to symplectic forms on $\mathcal{H}_p/N$
(or to non-degenerate classes in $H^2(\mathcal{H}_p/N,\C^*)$).
Hence, each one of them is extended to a pre-image of a Lagrangian in $\mathcal{H}_p/N$ and then induced to
$\mathcal{H}_p$.
By that we obtain $p-1$ representations of dimension $p=\sqrt{|\mathcal{H}_p/N|}$.
This exhausts all the irreducible representations of the Heisenberg group.

While the procedure of inducing the representation from a pre-image $H$ of a maximal isotropic subspace
$H/N$ to the group $G$ can easily be fitted for any abelian group $G/N$ (not necessarily a vector space), the
non-commutative case is fairly subtle. As the following theorem suggests, when $G/N$ is non-abelian,
one needs to demand normality of the maximal isotropic subgroup $H/N$ so as to induce an irreducible $G$-representation.
\begin{theorem}\cite{GOS}
Let $N\vartriangleleft G$, let $\eta\in$ Hom$(N,\C^*)^G$,
and let $H/N\vartriangleleft G/N$ be a maximal isotropic subgroup with respect to the class tra$(\eta)\in H^2(G/N,\C^*)$.
Then for every extension $\eta_0\in$ Hom$(H,\C^*)$, the $G$-representation ind$_H^G(\eta_0)$ is irreducible.
If, additionally, $|G/H|=|H/N|$ (and then $H/N$ is a Lagrangian with respect to tra$(\eta)$),
then the lifting ind$_H^G(\eta_0)$ depends only on $\eta$.
\end{theorem}
In \cite{GOS} an example is given where $H/N$ is maximal isotropic
with respect to a non-degenerate class tra$(\eta)\in H^2(G/N,\C^*)$. This subgroup is not normal in $G/N$,
and the representation ind$_H^G(\eta_0)$ ($\eta_0\in$ Hom$(H,\C^*)$ an extension of $\eta$) is reducible.
\subsection{}\label{indu}
In order to deal with Question \ref{q},
we go back to the vector space setup and present a standard way to construct an isotropic
subspace for an alternating bilinear form
$$\alpha:V\times V \rightarrow F.$$
Let $W:=$Span$_F\{v\}$, where $v\in V$ is any non-zero element, let $\omega:V\to V/W$ be the natural projection
and let $U\subset V$ be any complement of $W$.
Then with the notation of \eqref{pialpha} (clearly, the subspace $W$ is isotropic with respect to $\alpha$),
the subspace ker$(\pi_{\alpha})\subset V$ is either $V$ if $v$ is in
the radical of $\alpha$, or of codimension 1 otherwise. Then it is not hard to check that the restriction of the form
$\alpha$ to ker$(\pi_{\alpha})\subset V$ is inflated from an alternating bilinear form
$\alpha':V'\times V' \rightarrow F$, where $V':=\omega($ker$(\pi_{\alpha}))\subset V/W$.
Moreover, if the form $\alpha$ on $V$ is symplectic, then so is the form
$\alpha'$ on a space $V'$ whose dimension is strictly smaller than dim$_F(V)$.
Assume by induction on the vector space dimension that $\alpha'$ admits
an isotropic subspace $W'\subset V'$. Then the subspace
$\omega^{-1}(W')\subset V$ is an isotropic subspace with respect to $\alpha$. Further, if $W'$ is maximal, then so is
$\omega^{-1}(W')$.
In case $\alpha$ is symplectic, the various choices of
a 1-dimensional space $W$ and its complement $U$ (and by that the choice of $\pi_{\alpha}$) at each stage yield all
the Lagrangians for $\alpha$.

\subsection{} Our first concern in this paper, after a short preliminary \S\ref{conj},
is to imitate the above procedure (\S\ref{indu}) in a group-theoretic setup.
At every stage we need to choose a central element and hence we assume that the groups
we are dealing with are nilpotent or, rather,
$p$-groups. Such groups can help us gain a better
understanding of the general picture. Indeed, if $G$ is a group of central type with $[c]\in H^2(G,\C^*)$ non-degenerate,
then the restriction of $[c]$ to any $p$-Sylow subgroup of $G$ is also non-degenerate \cite[Lemma 2.7]{P}.

The group-theoretic analogue of the procedure in \S\ref{indu}, is described in \S\ref{Lagrangian} and
yields a positive answer to the weaker part (relaxing the normality condition) of Question \ref{q}:
\begin{proposition}\label{notnormal}
Let $[c]\in H^2(G,\C^*)$ be any class, where $G$ is nilpotent with $|G|\geq n^2$.
Then $[c]$ admits an isotropic subgroup $H<G$ with $|H|=n$.
In particular, non-degenerate classes over nilpotent groups admit a Lagrangian.
\end{proposition}
\begin{remark}
Both Propositions \ref{notnormal1} and \ref{notnormal} may be deduced from, e.g., \cite[Main Theorem]{h88}.
\end{remark}
When normality is imposed on isotropic subgroups,
the group-theoretic procedure in \S\ref{Lagrangian} yields a lower bound for the order of such groups
(Proposition \ref{smallisotropic}).
This bound retrieves the desired value $\sqrt{|G|}$ when the order of the nilpotent group $G$ is
free of eighth powers. Such groups give a positive answer to Question \ref{q}:
\begin{theorem}\label{positive}
Let $[c]\in H^2(G,\C^*)$ be a non-degenerate class, where $G$ is a nilpotent group whose order is free of 8-th powers.
Then $[c]$ admits a normal Lagrangian.
Consequently, any nilpotent group of central type $G$ of such order admits a short exact sequence \eqref{exten},
where $A$ is abelian of order $\sqrt{|G|}$, and $Q$ is an IYB group.
\end{theorem}

Nevertheless, Question \ref{q} does not always have a positive answer even for nilpotent groups.
In \S\ref{nolagrangian} we exhibit non-degenerate
classes over a family of metabelian $p$-groups of central type. The first use of this family
is to show the following interesting behavior.
\begin{proposition}\label{noextension}
For all primes $p\geq 3$ there exist non-degenerate classes $[c]\in H^2(G,\C^*)$ over
a family of groups $G$ of order $p^8$, which admit both
\begin{enumerate}
\item a normal Lagrangian $H\vartriangleleft G$ (of order $p^4$), and
\item a normal isotropic subgroup of order $p^3$ which is not contained in any normal Lagrangian.
\end{enumerate}
\end{proposition}
The second use is for a negative answer to Question \ref{q}.

\begin{proposition}\label{nolag}
For all primes $p\geq 5$ there exist non-degenerate classes $[c]\in H^2(G,\C^*)$ over
a family of groups $G$ of order $p^{12}$ which admit no normal Lagrangian.
\end{proposition}
We remark that it is fairly easy to find examples of non-degenerate classes $[c]\in H^2(G,\C^*)$
which admit no Lagrangian at all
(by Proposition \ref{notnormal}, $G$ cannot be nilpotent).
For example, consider the group $G=\Z_3\times (\Z_3\ltimes (\Z_2\times\Z_2))$ of order 36,
where the action of $\Z_3$ is given by its embedding in
Aut$(\Z_2\times \Z_2)$.
It turns our that this non-nilpotent group is of central type.
However, it can easily be seen that this group has no normal subgroup of order 6.

\end{section}

\section{Conjugation in twisted group algebras}\label{conj}
Non-degeneracy has a useful interpretation in the language of twisted group algebras.
We briefly recall the structure of such an algebra $\C^cG$. Let $c\in Z^2(G,\C^*)$ be a 2-cocycle. Then
$$\C^cG={\rm Span}_{\C}\{U_g\}_{g\in G},$$
where the multiplication is defined by
$$U_g\cdot U_h=c(g,h)\cdot U_{gh},\ \ g,h\in G.$$
For any subgroup $H<G$, the restriction of the 2-cocycle $c$ to $H$ gives rise to a sub-twisted group algebra
Span$_{\C}\{U_g\}_{g\in H}$ which we denote by $\C^cH$.

By a generalization of Maschke's theorem, $\C^cG$ is always semisimple.
That is $$\C^cG\simeq \oplus_{i=1}^rM_{n_i}(\C).$$
The projections of $\C^cG$ onto the matrix algebras $M_{n_i}(\C)$ $(1\leq i\leq r)$ correspond to the
irreducible $c$-representations of $G$ (of dimensions $n_i\leq \sqrt{|G|}$).
Furthermore, $r=$dim$_{\C}Z(\C^cG)$
is equal to the number of conjugacy classes of {\it regular} elements, that is the elements $g\in G$
such that with the notation of \eqref{perp}, $K_g(\alpha_c)=C_G(g)$
\cite[Theorem 2.4]{NvO}. In particular,
\begin{lemma}\label{tfae}
Let $c\in Z^2(G,\C^*)$. Then the following are equivalent.
\begin{enumerate}
\item $c$ is non-degenerate.
\item $G$ admits a unique irreducible $c$-representation (up to equivalence).
\item $G$ admits an irreducible $c$-representation of dimension $\sqrt{|G|}$.
\item dim$_{\C}Z(\C^cG)=1$.
\item $\C^cG$ is simple.
\end{enumerate}
\end{lemma}
For every normal subgroup $A\lhd G$, there is an action of $G$ on the sub-twisted group algebra
$\C^cA$ by conjugation:
\begin{equation}\label{act}
g(U_h)=U_g U_hU^{-1}_g=c(g,h)\cdot c^{-1}(ghg^{-1},g)\cdot U_{ghg^{-1}},\ \ g\in G, h\in A.
\end{equation}
In particular, $G$ acts on the set of central primitive idempotents of $\C^cA$.
\begin{lemma}\label{abelian}
Let $[c]\in H^2(G,\C^*)$ be a non-degenerate class and let $A\vartriangleleft G$ such that res$^G_A[c]=0$.
Then $A$ is abelian.
\end{lemma}
\begin{proof}
Let $E=\{e_1,\ldots,e_r\}$ be an orbit of the central primitive idempotents of $\C^cA$ under the above $G$-action.
Then $\sum_{j=1}^{r}e_{j}$ is a central idempotent in the simple algebra $\C^cG$,
and is hence equal to 1. Consequently, $E$ is the set of all central primitive idempotents of $\C^cA$.
The transitivity of the $G$-action implies that the simple components of $\C^cA$ are the same dimension.
Since the restriction of $[c]$ to $A$ is trivial, the twisted group algebra $\C^cA$ is isomorphic to the group
algebra $\C A$ and thus has a simple component of dimension 1.
It follows that all the irreducible representations of $A$ are 1-dimensional. This implies that $A$ is abelian.
\end{proof}
Let $[c]\in H^2(G,\C^*)$ be any cohomology class and let $A\vartriangleleft G$.
We describe the $G$-action by conjugation on the sub-twisted group algebra $\C^cA$
\begin{equation}\label{eta}
\eta:G\to \text{Aut}(\C^cA).
\end{equation}
By \eqref{act}, the algebra automorphism $\eta(g)$ sends a basis element $U_a \ \ (a\in A)$ to
~$c(g,a)\cdot c^{-1}(gag^{-1},g)\cdot U_{\phi_g(a)}$,
where $\phi_g(a)={gag^{-1}}$ is a group automorphism of $A$.
Then the image $\eta(G)$ embeds into the subgroup $\mathcal{G}_c$ of algebra automorphisms of $\C^cA$
which send each basis
element $U_a$ to a scalar multiple of $U_{\phi(a)}$ for some $\phi\in$Aut$(A)$.
Suppose that $g\in G$ induces the trivial group automorphism of $A$ by conjugation, that is $g\in \bigcap_{a\in A}C_G(a)$.
Then for every $a\in A$, $$U_gU_aU_g^{-1}=c(g,a)\cdot c^{-1}(a,g)\cdot U_{gag^{-1}}=\alpha_c(a,g)U_a,$$
where $\alpha_c$ is the alternating form associated to $c$
(see \eqref{altform}). Since $\alpha_c(-,g)\in\check{A}=$Hom$(A,\C^*)$,
it follows that $\mathcal{G}_c<$Aut$(\C^cA)$ fits into an exact sequence of the form
\begin{equation}\label{fits}
1\rightarrow B\rightarrow \mathcal{G}_c\rightarrow {\rm Aut}(A),
\end{equation}
where $B$ is a subgroup of $\check{A}$.
With the above notation, assume that $A$ is isotropic with respect to $[c]\in H^2(G,\C^*)$.
Then for any $g\in\ker(\eta)$, the subgroup $A\cdot\la g\ra<G$, generated by $A$ and $g$ is isotropic as well
(but not necessarily normal).
We make use of this observation in the sequel.

\begin{section}{Detecting Isotropic Subgroups}\label{Lagrangian}
\subsection{}
In this section we imitate the linear procedure described in \S\ref{indu} in group-theoretic terms.
By that we construct an isotropic subgroup $H<G$ (not necessarily normal) of order $\sqrt{|G|}$
with respect to any cohomology class $[c]\in H^2(G,\C^*)$ over a nilpotent group $G$.
This verifies Proposition \ref{notnormal} when $[c]$ is non-degenerate.
Clearly, since a finite nilpotent group $G$ is a direct product of its $p$-Sylow subgroups $P$,
then $ H^2(G,\C^*)=\oplus H^2(P,\C^*)$ and hence it is sufficient to assume that
$G$ itself is a $p$-group.
\begin{proposition}\label{nn}
Let $[c]\in H^2(G,\C^*)$ be any cohomology class over a $p$-group $G$ of order $\geq p^{2m}$.
Then $[c]$ admits an isotropic subgroup (not necessarily normal) of order $p^m$.
\end{proposition}
\begin{proof}
The proposition is proven by induction on $m$.
This is true of course for $m=1$
since in that case any subgroup of $G$ that is generated by an element of order $p$
is cyclic and hence isotropic by the vanishing of
the second cohomology of cyclic groups over $\C^*$.
Now, assume that the claim holds for a natural number $m$.
We prove that it is also true for $m+1$. Let $[c]\in H^2(G,\C^*)$ be any class,
where $G$ is a group of order $\geq p^{2(m+1)}$.
Take any central element $x\in Z(G)$ of order $p$ (recall that $p$-groups always have a nontrivial center).
The group homomorphism
$$\begin{array}{rcl}
c_x:G &\rightarrow &\C^*\\
 g&\mapsto & \alpha_c(g,x)=c(x,g)c(g,x)^{-1}.
\end{array}$$
takes its values in the subgroup of the complex $p$-th roots of 1, and hence its kernel $K_x$
(see \eqref{perp}) is either
of index 1 or $p$ in $G$. That is
\begin{equation}\label{indexp}
|K_x|\geq \frac{|G|}{p}.
\end{equation}
Since $x$ is central and res$^G_{K_x}[c]$-regular, then with the notation of \cite[Lemma 4.2]{R}
$\delta($res$^G_{K_x}[c])$ is trivial and hence
we deduce by \cite[Theorem 4.4]{R} that the restriction of $[c]$ to $K_x$
is in the image of the inflation from the quotient $K_x/\la x\ra$.
In other words, there exists a cohomology class $[c']\in H^2(K_x/\la x\ra, \C^*)$ such that
\begin{equation}\label{resinf}
{\rm res}^G_{K_x}[c]={\rm inf}^{K_x/\la x\ra}_{K_x}[c'].
\end{equation}
By \eqref{indexp} we have $$|K_x/\la x\ra|=\frac{|K_x|}{p}\geq \frac{|G|}{p^2}\geq p^{2m},$$
hence by assumption $[c']$ admits an isotropic subgroup $H< K_x/\la x\ra$ of order $p^m$.
Consequently, the pre-image of $H<G/\la x\ra$ under the natural projection $G\to G/\la x\ra$ is
$[c]$-isotropic in $G$ and of order $p^{m+1}$.
\end{proof}

\subsection{} In this section we give a lower bound on normal isotropic subgroups with respect to arbitrary
classes over nilpotent groups.
Again, it suffices to deal with $p$-groups.
\begin{proposition}\label{smallisotropic}
Let $G$ be a group of order $p^{m}$, $p$ a prime number, and let $[c]\in H^2(G,\C^*)$ be any cohomology class.
Then $[c]$ admits an isotropic normal abelian subgroup $A\lhd G$ of order $p^i$, for any $i$ such that
\begin{equation}\label{i}
\frac{i^2+i-2}{2}<m.
\end{equation}
\end{proposition}
\begin{proof}
We prove a bit stronger result, namely, as long as
$i$ satisfies \eqref{i}, then any normal abelian isotropic subgroup $A_{i-1}$ of order $p^{i-1}$
is contained in a normal abelian isotropic
subgroup of order $p^i$. This is done by induction on $i$ satisfying \eqref{i}.
For $i=1$, the only subgroup $A_{i-1}$ of order $p^{i-1}=1$ is the trivial subgroup $\{e\}<G$, which is clearly isotropic.
The condition $m>\frac{i^2+i-2}{2}=0$ says that $G$ is nontrivial and hence admits a central element $x\in G$ of order $p$.
Then the subgroup $\la x\ra <G$ is an isotropic (being cyclic) normal group of
order $p^1$ and contains  $A_{i-1}$. Assume now that $A_{i-1}\vartriangleleft G$ is isotropic abelian of order $p^{i-1}$.
Then $G$ acts on $\C^cA_{i-1}\simeq \C A_{i-1}$ by conjugation via $\eta:G\to \text{Aut}(\C^cA_{i-1})$ (see \eqref{eta}).
The order of the image $\eta(G)$ can be bounded using the exact sequence \eqref{fits}.
Since $|A_{i-1}|= p^{i-1}$, we have $|\check{A}_{i-1}|=p^{i-1}$ and
$|$Aut$(A_{i-1})|_{p}\leq p^{\frac{(i-1)(i-2)}{2}}$ \cite[\S 1.3]{H}.
By \eqref{fits}
\begin{equation}\label{bound}
|\mathcal{G}_c|\leq p^{i-1+\frac{(i-1)(i-2)}{2}}=p^{\frac{i(i-1)}{2}}.
\end{equation}
Now, $A_{i-1}$ is abelian, and hence it trivially conjugates itself.
Therefore, its image under $\eta$ is actually inside $\check{A}_{i-1}$,
and since it is also isotropic with respect to $c$, it follows that $A_{i-1}\vartriangleleft$ker$(\eta)$,
and $\eta$ thus factors through $$\bar{\eta}:G/A_{i-1}\to \text{Aut}(\C^cA_{i-1}).$$
By \eqref{i} and \eqref{bound} we have
$$|\text{Im}(\bar{\eta})| \leq |\mathcal{G}_c| \leq p^{\frac{i^2-i}{2}}< p^{m-i+1}=\frac{|G|}{|A_{i-1}|}= |G/A_{i-1}|,$$
and so $\bar{\eta}$ has a nontrivial kernel.
Next, since the center of a $p$-group nontrivially intersects any nontrivial normal subgroup,
there exists a central subgroup $\bar{A}<Z(G/A_{i-1})\cap {\rm Ker}(\bar{\eta})$ of order $p$.
Let $A_{i}\lhd G$ be the pre-image of $\bar{A}\lhd G/A_{i-1}$ under the natural projection.
Then it is clear that $A_{i}$ is of order $p^i$ and contained in Ker$(\eta)$.
The latter fact says that $A_i$ trivially conjugates $A_{i-1}$,
and hence it is abelian and induces the trivial character
in $\check{A}_{i-1}$. Consequently $A_{i}$ is isotropic with respect to $c$.\end{proof}

When $m<8$, Proposition \ref{smallisotropic} is already enough to obtain a normal
isotropic subgroup whose order is $\sqrt{|G|}$.
We record this fact as
\begin{corollary}\label{cor}
Let $[c]\in H^2(G,\C^*)$ be any cohomology class, where $G$ is of order $p^{2n}, n<4$.
Then $[c]$ admits a normal isotropic subgroup of order $p^n$.
\end{corollary}
Theorem \ref{positive} is a consequence of this corollary,
since any nilpotent group is the direct product of its $p$-Sylow subgroups.
\end{section}

\begin{section}{Examples}
In this section we construct non-degenerate classes over a family of groups of central type of the form
$G_n=(\mathbb{Z}_p)^{2n}\rtimes(\mathbb{Z}_p\times\mathbb{Z}_p)$.
In \S\ref{3} we show that the classes over $G_3$ admit a normal isotropic subgroup which is not contained in any normal
Lagrangian (but do admit a normal Lagrangian).
In \S\ref{5} we show that the classes over $G_5$ do not admit a normal Lagrangian at all.
These prove Propositions \ref{noextension} and \ref{nolag} respectively.

Let $A=A_n$ be an $n$-dimensional vector space over the prime field $\mathbb{F}_p$,
where $n$ is some natural number such that $2<n\leq p$.
As an abelian group $A$ is isomorphic to $(\mathbb{Z}_p)^n$, the elementary abelian $p$-group of rank $n$.
Considering $A$, and hence also $\check{A}$, as trivial $\check{A}$-modules, the identity morphism
$$\pi:\check{A}\to \check{A}$$ is clearly a 1-cocycle. Let $H=H_n=A\times \check{A}$.
Then the 2-cocycle $c_{\pi}=c_{n,\pi}\in Z^2(H,\C^*)$ induced by $\pi$ (see \eqref{eg}) has the form
$$c_{\pi}(a_1\chi_1, a_2\chi_2)=\la\chi_1,a_2\ra ^{-1},\ \ a_1,a_2\in A,\ \ \chi_1,\chi_2\in \check{A}.$$
Since $\pi$ is bijective, it follows that $c_{\pi}$ is non-degenerate.
By Lemma \ref{tfae}, the twisted group algebra $\C^{c_{\pi}}H$ is simple,
and there exists a unique irreducible $c_{\pi}$-projective representation of $H$,
which is of dimension $p^n$.
This irreducible representation
$$\nu:\C^{c_{\pi}}H\xrightarrow{\thicksim} \text{End}_{\C}(W).$$
is given explicitly by identification of $\C^{c_{\pi}}H$
with an endomorphism algebra of the following linear space $W$.
Let $W=$Span$_{\C}\{w_a\}_{a\in A}$ be the regular representation of $\C A$, that is
\begin{equation}\label{acta}
U_{a'}\cdot w_a=w_{a'a}
\end{equation} for every
$a,a'\in A$. We also define for every $\chi\in \check{A}$ and $a\in A$
\begin{equation}\label{actchi}
U_{\chi}\cdot w_{a}= \la \chi, a\ra^{-1}w_{a},
\end{equation}
Now, endow $W$ with a $\C^{c_{\pi}}H$-module structure by imposing $U_{a\chi}:=U_aU_{\chi}$ for every $a\chi\in H$.

The next step is to extend $c_{\pi}$ to a non-degenerate 2-cocycle over a group $G_n$ of order $p^{2n+2}$,
which contains $H$ as a normal subgroup.
For that we define the following action of $\mathbb{Z}_p\times \mathbb{Z}_p$ on $H$.
Choose any $\mathbb{F}_p$-basis $\{x_1,\cdots, x_n\}$ of $A$, define
\begin{eqnarray}\begin{array}{rcl}\label{T}
T=T_n:A&\rightarrow &A\\
x_i& \mapsto & \left \{
\begin{array}{cl}
0 &  i=1\\
x_{i-1} & 1<i\leq n,
\end{array}\right.
\end{array}\end{eqnarray}
and extend it $\mathbb{F}_p$-linearly.
With respect to the above basis, $T$ is a nilpotent Jordan block.

Note that the commuting linear operators $R=R_n := 1+T$ and $S=S_n := 1+T^2$ are invertible of order $p$
(here we use the fact that $p\geq n$).
The actions of $R$ and $S$ on $A$ induce the diagonal actions \eqref{diag} on $\check{A}$,
which can be extended in a natural way to an action on $H$ as follows.
\begin{equation}\label{actH}
R(a\chi) := R(a)R(\chi),\ \ S(a\chi) := S(a)S(\chi),\ \ a\in A, \chi\in \check{A}.
\end{equation}
With the action \eqref{actH} we define $$G=G_n:= H\rtimes \la R,S\ra\simeq H\rtimes (\mathbb{Z}_p\times \mathbb{Z}_p). $$
In order to extend the 2-cocycle $c_{\pi}$ to $G$,
note that it is invariant under the action of $\la R,T\ra$ on $Z^2(H,\C^*)$ induced from its action on $H$ \eqref{actH}
(for the definition of the action see \cite[\S III8]{b}).
Indeed, for any $a_1,a_2\in A$ and $\chi_1,\chi_2\in \check{A}$ we have
$$\begin{array}{l}
R(c_{\pi})(a_1\chi_1, a_2\chi_2)=c_{\pi}(R^{-1}(a_1\chi_1), R^{-1}(a_2\chi_2))
=c_{\pi}(R^{-1}(a_1)R^{-1}(\chi_1), R^{-1}(a_2)R^{-1}(\chi_2))\\
=\la R^{-1}(\chi_1),R^{-1}(a_2)\ra ^{-1}
=\la\chi_1,a_2\ra ^{-1}=c_{\pi}(a_1\chi_1, a_2\chi_2),
\end{array}$$
and similarly for $S$.

Let $\zeta=\zeta_p$ be a nontrivial $p$-th root of unity. For any $0\leq r \leq p-1$, define
\begin{eqnarray}\label{extcoc}
\begin{array}{rcl}
c_{\pi}^{(r)}:G\times G &\to & \C^*\\
(h_1R^iS^j,h_2R^kS^l) &\mapsto & c_{\pi}(h_1,R^iS^j(h_2))\zeta^{rjk},\ \ h_1,h_2\in H.
\end{array}
\end{eqnarray}
Using the fact shown above that $R$ and $S$ preserve $c_{\pi}$ as a function,
it is not hard to check that $c_{\pi}^{(r)}\in Z^2(G,\C^*)$ for every $0\leq r \leq p-1$.

We describe the $G$-action by conjugation on the sub-twisted algebra $\C^{c_{\pi}}H$.
Conjugations by $U_R$ and by $U_S$ are automorphisms of the matrix algebra $\C^{c_{\pi}}H$,
and therefore they are the same as conjugation by invertible matrices $M_R$ and $M_S$ respectively.
The matrices $M_R$ and $M_S$ are uniquely determined, up to a scalar, by their
actions on the basis $\{ w_a\}_{a\in A}$ of $W$ as follows.
\begin{equation}\label{RS}
M_R\cdot w_a=w_{R(a)},\ \ M_S\cdot w_a=w_{S(a)},\ \ a\in A.
\end{equation}
Note that the operators $M_R$ and $M_S$ commute.
It is also important to notice that $M_R$ and $M_S$
(or, rather, their scalar multiples) are independent of $0\leq r\leq p-1$.
We claim
\begin{lemma}
The cocycles $c^{(r)}_{\pi}$ are non-degenerate for all $r\neq 0$.
\end{lemma}
\begin{proof}
Let $V$ be an irreducible representation of $\C^{c_{\pi}^{(r)}}G$ for some $0\leq r\leq p-1$.
Then as a $\C^{c_{\pi}}H$-module $V\simeq \oplus W$, where $W$ is the unique irreducible $\C^{c_{\pi}}H$-module
(determined by \eqref{acta} and \eqref{actchi}).
Since $V$ is irreducible, the dimension of $V$ is a power of $p$ and does not exceed $\sqrt{|G|}=p^{n+1}$.
Since it is isomorphic to a direct sum of copies of the $p^n$-dimensional space $W$ we obtain dim$_{\C}V$
is either $p^n$ or $p^{n+1}$.

 Assume that dim$_{\C}V=p^n$. This means that the restriction of the irreducible $\C^{c_{\pi}^{(r)}}G$-representation
$V$ to $\C^{c_n}H$ is exactly $W$. In other words, there exist a morphism
$$\nu':\C^{c_{\pi}^{(r)}}G\to \text{End}_{\C}(W)$$ extending $\nu$.
Since $\nu'$ is irreducible and hence surjective when restricted to $\C^{c_{\pi}}H$, it follows by \eqref{RS} that
a necessary condition for $\nu'$ to be a homomorphism is that
$\nu'(U_R)=t_RM_R$ and $\nu'(U_S)=t_SM_S$ for some nonzero scalars $t_R$ and $t_S$
(we use the symbols $U_R$ and $U_S$ for all $0\leq r\leq p-1$).
Now,
$$I_n=M_RM_SM_R^{-1}M_S^{-1}=\nu'(U_RU_SU_R^{-1}U_S^{-1})=\nu'(\zeta^r)=\zeta^rI_n,$$
implying that $r=0$.
Consequently, if $r\neq 0$ then dim$_{\C}V=p^{n+1}=\sqrt{|G|}$. Again by Lemma \ref{tfae},
the cocycles $c^{(r)}_{\pi}$ are non-degenerate for all $r\neq 0$.
\end{proof}
\begin{remark}
For $r=0$, $\nu'$ is indeed a representation of $\C^{c_{\pi}^{(0)}}G$ of dimension $p^n$ by
choosing $t_R=t_S=1$. Then ${c_{\pi}^{(0)}}$ is degenerate.
\end{remark}
\subsection{Proof of Proposition \ref{noextension}}\label{nolagrangian}\label{3}
Certainly, $A$ is an isotropic normal subgroup of $G$ of order $p^n$ with respect to all cocycle $c_{\pi}^{(r)}$.
However, $A$ is not contained in any normal Lagrangian of $G$ with respect to the non-degenerate cocycles
$c_{\pi}^{(r)}, r\neq 0$.
Indeed, by Lemma \ref{abelian} any Lagrangian must be abelian.
However, $U_g$ acts non-trivially on $\C^{c^{(r)}_{\pi}|_A}A$ for all $g\in G\setminus A$
(here we use the assumption that $2<n$, as otherwise $S$ acts trivially on $A$).
In order to complete the proof of Proposition \ref{noextension},
we show that for $n=3$, the cocycles $c^{(r)}_{\pi}$ do admit a normal Lagrangian
(which does not contain the isotropic normal subgroup $A$) for all $r\neq 0$.
With the basis that defines $R$ and $S$ \eqref{T}, let $\chi_3$ be the following character in $\check{A}$.
$$\begin{array}{rcl}
\chi_3: A&\to& \C^*\\
x_1, x_2& \mapsto &1\\
x_3& \mapsto &\zeta_3.
\end{array}$$
Then, as can easily be checked, $L=\langle x_1,x_2,\chi_3,S\rangle$ is a normal Lagrangian with respect to $c^{(r)}_{\pi}$
for all $r\neq 0$.

\subsection{Proof of Proposition \ref{nolag}}\label{nolagrangian}\label{5}
As for the case $n=5$, we show that a normal Lagrangian of $G_5$ does not exist.
Fix $0\leq r \leq p-1$ and assume, by negation, that $c_{\pi}^{(r)}$ admits a isotropic subgroup $L\vartriangleleft G$
of order $p^6$.
The subgroup $L\cap H$ is an isotropic subgroup of $H$ with respect to the non-degenerate cocycle $c_{\pi}$,
and by Proposition \ref{notnormal},
\begin{equation}\label{divisor}
|L\cap H|\leq p^5.
\end{equation}
This means that there exists an element $g\in L\setminus H$.
Write $g=hR^{l_1}S^{l_2}$, where $l_1$ and $l_2$ do not both divide $p$.
Normality of $L$ implies $[H,g]<L$. For any $a\chi\in H$,
$[a\chi,g] = a\chi R^{l_1}S^{l_2}(a\chi)^{-1}$ under the action \eqref{actH}.
But as an ${\mathbb{F}_p}$-linear operator,
 $$\text{rank}R^{l_1}S^{l_2} - I\geq 3$$ for any nontrivial $(l_1,l_2)$. Therefore
$$\dim_{\mathbb{F}_p}L\cap H\geq 3+3 = 6$$
contradicting \eqref{divisor}.

\end{section}

\noindent{\bf Acknowledgement.}
We thank P. Singla and U. Onn for their helpful suggestions.
The third author was supported by the Danish National Research
Foundation (DNRF) through the Centre for Symmetry and Deformation.

\end{document}